\documentclass[amstex,reqno,10pt]{amsart}
\usepackage{amsmath,amsfonts,amssymb,amsthm,enumerate,hyperref,multicol, graphicx, wrapfig, pdfsync,array, caption, subcaption, color,amsaddr}
\usepackage[dvipsnames]{xcolor}
\usepackage{epstopdf}
\newtheorem{lemma}{Lemma}
\newtheorem{thm}{Theorem}
\newtheorem{definition}{Definition}
\newtheorem{example}{Example}
\newtheorem{corollary}{Corollary}

\numberwithin{equation}{section}
\begin{document}
	\leftline{ \scriptsize \it}
\title[]
{Fractional $\alpha$-Bernstein-Kantorovich operators of order $\beta$: A new construction and approximation results}
\maketitle
\begin{center}
	{\bf  Jaspreet Kaur$^{1,2,a}$, Meenu Goyal$^{2,b}$, Khursheed J. Ansari$^{3,c}$}
	\vskip0.2in
	$^{1}$ Department of Mathematics, GLA University, Mathura-281406, India\\
	$^{2}$Department of Mathematics,
	Thapar Institute of Engineering and Technology, Patiala-147004, India\\
	$^3$Department of Mathematics, College of Science, King Khalid University, 61413, Abha, Saudi Arabia
	\vskip0.2in
	${}^{a}$jazzbagri3@gmail.com, ${}^{b}$meenu\_rani@thapar.edu and ${}^{c}$ansari.jkhursheed@gmail.com
\end{center}
\begin{abstract}
In the current article, we establish a distinct version of the operators defined by Berwal \emph{et al.}, which is the Kantorovich type modification of $\alpha$-Bernstein operators to approximate Lebesgue's integrable functions. We define its modification that can preserve the linear function and analyze its characteristics. Additionally, we construct the bivariate of blending type operators by Berwal \emph{et al.}. We analyze both its the convergence and error of approximation properties by using the conventional tools of approximation theory. Finally, we demonstrate our results by presenting examples that highlight graphical visuals using MATLAB.

	\textbf{Keywords:}  Positive linear operators, Convergence, Modulus of continuity.\\
	\textbf{Mathematics Subject Classification(2020):} 26A15, 40A05, 41A10, 41A25, 41A35.
\end{abstract}
\section{Introduction}
Bernstein operators were introduced by S. N. Bernstein to give the proof of the very well-known Weierstrass approximation theorem. These operators are the most simplest and constructive operators with so many remarkable properties. Due to these properties, later these polynomials became famous in many computer applications.  These operators are useful to approximate continuous functions. These are extended in several directions. Kantorovich \cite{kan} introduced the classical Bernstein-Kantorovich operators, defined as
\begin{eqnarray}
K_n(f;x)&=&(n+1)\displaystyle\sum_{k=0}^n p_{n,k}(x)\int_{\frac{k}{n}}^{\frac{k+1}{n}}f(t)\,\mathrm{d}t\nonumber\\
&=&(n+1)\displaystyle\sum_{k=0}^n p_{n,k}(x)\int_{0}^1f\left(\dfrac{k+t}{n+1}\right)\,\mathrm{d}t,\quad\quad x\in[0, 1],~~n\in \mathbb{N},
\end{eqnarray}
where $f:[0, 1]\rightarrow \mathbb{R}$ is an integrable function and $p_{n,k}(x)$ are known as Bernstein basis polynomials defined as $$p_{n,k}(x)=\binom{n}{k}x^k(1-x)^{n-k},\quad k=0,1,\dotsb,n.$$
Also, another generalizations of these operators are still studied by many researchers as \cite{DT, two, GA}.  
Baskakov \cite{Baskakov1957} and Sz\'asz \cite{szasz1950general} introduced the positive linear operators to approximate the continuous functions on infinite interval. The integral operators associated with these operators can be seen in \cite{ansari2022kant,ansari2022num,bushnaq2022comp,  cai2022stat,rahman2024esti}.
In order to reduce the error of approximation of the operators, generalizations of the positive linear operators has been started as one can see \cite{ lambe, QBER, stan}.
Chen \textit{et al.} \cite{CHE} introduced the modification of Bernstein polynomials depending on a real parameter $\alpha \in [0,1]$, which are defined as
\begin{eqnarray}\label{alpha}
	T_{n,\alpha}(f;x)&=&\displaystyle \sum_{k=0}^n b_{n,k}^{\alpha}(x)f\left(\frac{k}{n}\right),
\end{eqnarray}
{where}
\begin{eqnarray}\label{alpha1}
	b_{n,k}^{\alpha}(x)&=&\left[{n-2\choose k}(1-\alpha)x+{n-2 \choose k-2}(1-\alpha)(1-x)+{n \choose k}\alpha x(1-x)\right]x^{k-1}(1-x)^{n-k-1},\, n\geq2.\nonumber\\
	&&\hspace{8cm}
\end{eqnarray}
For $n=1$, we have $b_{1,0}^{\alpha}(x)=1-x, \quad b_{1,1}^{\alpha}(x)=x.$ \\
The authors studied its shape-preserving as well as approximation properties.  In 2019, Deo and Pratap \cite{deo2020alpha} presented the Kantorovich variant of $\alpha$-Bernstein operators and studied the direct approximation theorem, and asymptotic results for these operators.  Kajla and Acar \cite{kajla2018blend} introduced the Durrmeyer modification of the summation operators \eqref{alpha} and studied the rate of convergence and some approximation properties. Kajla and Goyal \cite{kajla2019general} modified the Durrmeyer variant of these operators by using P\u{a}lt\u{a}nea basis function in an integral depending on a parameter $\rho>0.$ 

Baytun\c{c} et al. \cite{baytuncc2023app} generalize the Bernstein-Kantorovich operators using Riemann-Liouville  type fractional integers. The authors studied the approximation properties of these operators. Recently, Berwal \emph{et al.} defined the Kantorovich type generalization of $\alpha-$Bernstein operators given by Deo and Pratap \cite{deo2020alpha} 
for $f\in C[0,1]$ and $ n\in \mathbb{N}$ in the following way:
\begin{eqnarray}\label{op}
K_{n,\alpha}^{\beta}(f;x)=\Gamma{(\beta+1)}\sum_{k=0}^n b_{n,k}^{\alpha}(x)\int_0^1\frac{(1-t)^{\beta-1}}{\Gamma{(\beta)}}f\left(\frac{k+t}{n+1}\right)\, \mathrm{d}t,
\end{eqnarray}	
where $\alpha\in [0,1], \beta\in (0,1]$ and using the following definitions:

\begin{definition}
	The Gamma function also known as Euler's Gamma function denoted by $\Gamma:(0,\infty)\rightarrow \mathbb{R}$ is defined as
	$$\Gamma(\beta)=\int_0^{\infty}t^{\beta-1}e^{t-1}\, \mathrm{d}t, \quad \text{Re}(\beta)>0,$$
	 and by using analytic continuation for all $\beta\in \mathbb{C}$ except the simple poles at the non-positive integers.
\end{definition}
\begin{definition}
	The beta function of find kind is defined by
	$$B(x,y)=\int_0^1t^{x-1}(1-t)^{y-1}\, \mathrm{d}t\quad Re(x), Re(y)>0.$$
	The Gamma and beta functions are related by the following relationship:
	$$B(x,y)=\dfrac{\Gamma(x)\Gamma(y)}{\Gamma(x+y)} \quad x,y>0.$$
\end{definition}
\begin{definition}\label{def4}\cite{kilbas1993}
	The Riemann-Liouville type fractional integral $(I_{a^+}^\beta f)(x)$ of order $\beta>0, (Re(\beta)>0)$ for integrable function $f\in[0,\infty)$ is defined as
	$$({}_{a^+}I_x^{\beta}f)(.)=\dfrac{1}{\Gamma(\beta)}\int_a^x(x-t)^{\beta-1}f(t)\,\mathrm{d}t, \quad x>0.$$
\end{definition}
 Motivated by these generalizations, we present the new positive linear operators using Riemann-Liouville type fractional integrals that can preserve linear functions and its generalization to approximate the functions of two variables.

The present article is structured as:  In the section \ref{affine}, we define a modification of the operators \eqref{op} introduced by Berwal \emph{et al.} \cite{Berwalapp} in order to preserve the linear functions as well as constants. We study the convergence and linear precision property of these operators. In the next section \ref{bivariate}, we define the bivariate of the operators and study its results. The last section \ref{numerical} is equipped with some examples to justify the results of the operators and the effect of the different introduced parameters.
			\section{Preserving affine functions}\label{affine}
			The well-known Bernstein operators satisfy the linear precision property, but the operators \eqref{op} do not satisfy this property. So, we modify the operators to get this property. Therefore, we construct the operators in the following form:
			\begin{eqnarray}\label{op1}
				A_{n,\alpha}^{\beta}(f;x)=\Gamma(\beta+1)\sum_{k=0}^n b_{n,k}^{\alpha}(x)\int_0^1 \frac{(1-t)^{\beta-1}}{\Gamma(\beta)}f\left(a_{n,k}\frac{k+t}{n+1}\right)\, \mathrm{d}t,
			\end{eqnarray}
			where we need to choose $a_{n,k}$ in such a way that $A_{n,\alpha}^{\beta}(e_i;x)=e_i, \,\, i=0,1.$ Thus, by the moments of $\alpha$-Bernstein operators, we require
\begin{enumerate}
	\item $\displaystyle\Gamma(\beta+1)\int_0^1 \dfrac{(1-t)^{\beta-1}}{\Gamma(\beta)}\, \mathrm{d}t=1,$
	\item $\displaystyle\Gamma(\beta+1)\int_0^1 \dfrac{(1-t)^{\beta-1}}{\Gamma(\beta)} a_{n,k}\dfrac{k+t}{n+1}\, \mathrm{d}t=\dfrac{k}{n}.$	
	\end{enumerate}		
	By solving these equations, we get $a_{n,k}=\dfrac{k(n+1)(\beta+1)}{n[k(\beta)+1]}.$\\
Now we have the lemma as follows:

			\begin{lemma}\label{lemma1}
We can compute the moments of our operators $A_{n,\alpha}^{\beta}(f;x)$ as follows:
					\item $A_{n,\alpha}^{\beta}(e_0;x)=1;$
\item $A_{n,\alpha}^{\beta}(e_1;x)=x;$
\item $A_{n,\alpha}^{\beta}(e_2;x)=x^2+\dfrac{n+2(1-\alpha)}{n^2}x(1-x)+\dfrac{\beta}{2+\beta}T_{n,\alpha}\left(\dfrac{t^2}{(nt(\beta+1)+1)^2};x\right).$
			\end{lemma}
			\begin{thm}
				For any $f\in C[0,1], \beta\in(0,1],$ we have $A_{n,\alpha}^{\beta}(f;x)$ converges uniformly to $f(x).$
			\end{thm}
			\begin{proof}
				Since $A_{n,\alpha}^{\beta}(e_0;x)=1, A_{n,\alpha}^{\beta}(e_1;x)=x.$ Now, in order to prove the theorem, we need to show that $\displaystyle\lim_{n\rightarrow \infty}A_{n,\alpha}^{\beta}(e_2;x)=x^2.$\\
				Consider $\mu_{n,k}(t)=\dfrac{t^2}{(nt(\beta+1)+1)^2}.$
				It is easy to check that $\displaystyle\lim_{n\rightarrow \infty}\mu_{n,k}(t)=0.$
Thus,
$$A_{n,\alpha}^{\beta}(e_2;x)=x^2+\dfrac{n+2(1-\alpha)}{n^2}x(1-x)+\dfrac{\beta}{2+\beta}T_{n,\alpha}\left(\mu_{n,k}(t;x)\right).$$
Using these, we obtain $$\displaystyle\lim_{n\rightarrow\infty}A_{n,\alpha}^{\beta}(e_2;x)=x^2.$$
As $T_{n,\alpha}(f;x)\rightarrow f(x)$ as $n\rightarrow \infty.$
Hence, by using the Bohman-Korovkin theorem, we reach at the required result.
			\end{proof}
			\begin{corollary}
				The operators $A_{n,\alpha}^{\beta}(f;x)$ reproduces the linear polynomials that is $A_{n,\alpha}^{\beta}(at+b;x)=ax+b,$ where $a,b\in\mathbb{R}.$
			\end{corollary}
			\section{Bivariate of the operators}\label{bivariate}
			In practical scenarios containing multiple variables such as inventory levels, customer expenditures, market values of the stocks etc., the applications of mathematical methods are essential. Frequently, these models include functions those are dependent upon two or more variables. They can be seen in many other subjects as Fluid Dynamics, Economics, Continuum Mechanics, and Thermodynamics. This showcases the importance of the multi-variable functions that are used to understand and solve complex issues in various domains. These applications motivate the researchers to find the positive linear operators that can approximate such kinds of functions. 
			We define the bivariate of the operators \eqref{op} for $f\in C(I^2),$ where $I^2=[0,1]\times[0,1]$ as follows:
			
\begin{align}\label{op2}
&K_{n,m,\alpha_{1},\alpha_2}^{\beta_1,\beta_2}(f;x,y)\\
&=\Gamma(\beta_1+1)\Gamma(\beta_2+1)\sum_{k=0}^n \sum_{j=0}^m b_{n,k}^{\alpha_1}(x)b_{m,j}^{\alpha_2}(y)\int_0^1\int_0^1\frac{(1-t)^{\beta_1-1}(1-s)^{\beta_2-1}}{\Gamma(\beta_1)\Gamma(\beta_2)}f\left(\frac{k+t}{n+1}, \frac{j+s}{m+1}\right)\, \mathrm{d}s \,\mathrm{d}t,\nonumber
\end{align}
where $(x,y)\in I^2, \beta_1, \beta_2\in (0,1].$ When $\beta_1=\beta_2=1$ we get back the classical $\alpha$-Bernstein Kantorovich operators. Moreover, by taking $\alpha_1=\alpha_2=\beta_1=\beta_2=1,$ the operators reduced back to Bernstein-Kantorovich operators.
Due to the importance of the moments to find the results of positive linear operators, we first compute the moments for the operators \eqref{op2}.
\begin{lemma}\label{lemma3} The operators \eqref{op2} satisfy the following identities:
\begin{align*}
K_{n,m,\alpha_{1},\alpha_2}^{\beta_1,\beta_2}(e_{00};x,y)=&1;\\
K_{n,m,\alpha_{1},\alpha_2}^{\beta_1,\beta_2}(e_{10};x,y)=&\dfrac{nx}{n+1}+\dfrac{1}{(n+1)(\beta_1+1)};\\
K_{n,m,\alpha_{1},\alpha_2}^{\beta_1,\beta_2}(e_{01};x,y)=&\dfrac{my}{m+1}+\dfrac{1}{(m+1)(\beta_2+1)};\\
K_{n,m,\alpha_1,\alpha_2}^{\beta_1,\beta_2}(e_{20};x,y)=&\dfrac{n^2}{(n+1)^2}\left[x^2+\dfrac{n+2(1-\alpha_1)}{n^2}x(1-x)\right]\\
&+\dfrac{2nx}{(\beta_1+1)(n+1)^2}+\dfrac{2}{(n+1)^2(2+\beta_1)(1+\beta_1)};\\
K_{n,m,\alpha_1,\alpha_2}^{\beta_1,\beta_2}(e_{02};x,y)=&\dfrac{m^2}{(m+1)^2}\left[y^2+\dfrac{m+2(1-\alpha_2)}{m^2}y(1-y)\right]\\
&+\dfrac{2my}{(\beta_2+1)(m+1)^2}+\dfrac{2}{(m+1)^2(2+\beta_2)(1+\beta_2)};
\end{align*}
\end{lemma}

\begin{thm}\label{thm2}
For every $f\in C(I^2),$ the operators \eqref{op2} converges uniformly to $f(x,y).$
\end{thm}
			\begin{proof}
				The proof is straightforward by using Lemma \ref{lemma3} and Volkov's theorem. Hence, we omit the details.
			\end{proof}
			\begin{definition}\label{complete} (\cite{anastassiou2012approx}, p.80)
				For any $g\in C(I^2),$ the complete modulus of continuity is defined as
				\begin{eqnarray*}
					\omega(g;\delta_1,\delta_2)&=&\sup\{\mid g(u,v)-g(x,y)\mid:(u,v),(x,y)\in I^2 ,\,\mid u-x\mid\leq \delta_1,\,\mid v-y\mid\leq \delta_2\},\\
					\mbox{or} \quad \omega(g;\delta)
					&=&\sup\{\mid g(u,v)-g(x,y)\mid: \sqrt{(u-x)^2+(v-y)^2}\leq \delta, \,(u,v),(x,y)\in I^2\},
				\end{eqnarray*}
			\end{definition}
			\begin{thm}\label{complete}
				Let $f\in C(I^2)$ and $\beta_1,\beta_2\in (0,1].$ Then for all $(x,y)\in I^2,$ we have
				$$\mid K_{n,m,\alpha_{1},\alpha_2}^{\beta_1,\beta_2}(f;x,y)-f(x,y)\mid \leq 4\omega(f;\delta_{n,\alpha_1,\beta_1}^2(x),\delta_{m,\alpha_2,\beta_2}^2(y)),$$
				where $\delta_{n,\alpha_1,\beta_1}^2(x):=K_{n,m,\alpha_{1},\alpha_2}^{\beta_1,\beta_2}((t-x)^2;x,y)$ and $\delta_{m,\alpha_2,\beta_2}^2(y):=K_{n,m,\alpha_{1},\alpha_2}^{\beta_1,\beta_2}((s-y)^2;x,y).$
		 \end{thm}	
		 \begin{proof}
		 	Consider 
		 	\begin{eqnarray*}
		 		\mid K_{n,m,\alpha_{1},\alpha_2}^{\beta_1,\beta_2}(f;x,y)-f(x,y)\mid &\leq & K_{n,m,\alpha_{1},\alpha_2}^{\beta_1,\beta_2}(\mid f(t,s)-f(x,y)\mid;x,y)\\
		 		&\leq &K_{n,m,\alpha_{1},\alpha_2}^{\beta_1,\beta_2}\left(\left(1+\frac{\mid t-x \mid}{\delta_1}\right)\left(1+\frac{\mid s-y \mid }{\delta_2}\right)\omega(f;\delta_1,\delta_2);x,y\right)\\
		 		&\leq & \omega(f;\delta_1,\delta_2)\left[1+\frac{K_{n,m,\alpha_{1},\alpha_2}^{\beta_1,\beta_2}(\mid t-x\mid;x,y)}{\delta_1}\right]\\
		 		&&\times\left[1+\frac{K_{n,m,\alpha_{1},\alpha_2}^{\beta_1,\beta_2}(\mid s-y \mid;x,y)}{\delta_2}\right].
		 	\end{eqnarray*}
		 	Using the Cauchy-Schwarz inequality, we obtain
		 	$$K_{n,m,\alpha_{1},\alpha_2}^{\beta_1,\beta_2}(\mid t-x\mid;x,y)\leq \sqrt{K_{n,m,\alpha_{1},\alpha_2}^{\beta_1,\beta_2}((t-x)^2;x,y)},$$
		 	$${K_{n,m,\alpha_{1},\alpha_2}^{\beta_1,\beta_2}(\mid s-y \mid;x,y)}\leq \sqrt{K_{n,m,\alpha_{1},\alpha_2}^{\beta_1,\beta_2}((s-y)^2;x,y)}.$$
Therefore 
\begin{align*}
    &\mid K_{n,m,\alpha_{1},\alpha_2}^{\beta_1,\beta_2}(f;x,y)-f(x,y)\mid \\
    &\qquad\qquad\leq \omega(f;\delta_1,\delta_2) \left(1+\frac{K_{n,m,\alpha_{1},\alpha_2}^{\beta_1,\beta_2}((t-x)^2;x,y)}{\delta_1}\right)\left(1+\frac{K_{n,m,\alpha_{1},\alpha_2}^{\beta_1,\beta_2}((s-y)^2;x,y)}{\delta_2}\right).
\end{align*}
		 	
		 	Now, by choosing $\delta_1=\delta_{n,\alpha_1,\beta_1}^2(x)$ and $\delta_2=\delta_{m,\alpha_2,\beta_2}^2(y),$ we get the required inequality.
		 \end{proof}
		 \begin{definition}\label{partial} (\cite{anastassiou2012approx}, p.81)
		 	For any $g\in C(I^2)$, the partial modulus of continuity is defined by
		 	\begin{eqnarray*}
		 		\omega^1(g;\delta_1)&=&\sup\{\mid g(x_1,y)-g(x_2,y)\mid: y\in I,\mid x_1-x_2\mid \leq \delta_1\},\\
		 		\mbox{and}\quad
		 		\omega^2(g;\delta_2)&=&\sup\{\mid g(x,y_1)-g(x,y_2)\mid: x\in I, \mid y_1-y_2 \mid \leq \delta_2\}.
		 	\end{eqnarray*}
		 \end{definition}
		 \begin{thm}\label{partial}
		 		Let $f\in C(I^2)$ and $\beta_1,\beta_2\in (0,1].$ Then for all $(x,y)\in I^2,$ we have
		 	$$\mid K_{n,m,\alpha_{1},\alpha_2}^{\beta_1,\beta_2}(f;x,y)-f(x,y)\mid \leq 2[\omega^1(f;\delta_{n,\alpha_1,\beta_1}^2(x))+\omega^2(f;\delta_{m,\alpha_2,\beta_2}^2(y))],$$
		 	where $\delta_{n,\alpha_1,\beta_1}^2(x)$ and $\delta_{m,\alpha_2,\beta_2}^2(y)$ are as defined in Theorem \ref{complete}.
		 \end{thm}
		 \begin{proof}
		 	Using the linearity of the operators, we get
\begin{align*}
&\mid K_{n,m,\alpha_{1},\alpha_2}^{\beta_1,\beta_2}(f;x,y)-f(x,y)\mid\\
&\quad\leq K_{n,m,\alpha_{1},\alpha_2}^{\beta_1,\beta_2}(\mid f(t,s)-f(x,s)\mid;x,y)+K_{n,m,\alpha_{1},\alpha_2}^{\beta_1,\beta_2}(\mid f(x,s)-f(x,y)\mid;x,y)\\
&\quad\leq K_{n,m,\alpha_{1},\alpha_2}^{\beta_1,\beta_2}(\omega^1(f;\mid t-x \mid);x,y)+K_{n,m,\alpha_{1},\alpha_2}^{\beta_1,\beta_2}(\omega^2(f;\mid s-y \mid);x,y)\\
&\quad \leq  \left(1+\frac{\sqrt{K_{n,m,\alpha_{1},\alpha_2}^{\beta_1,\beta_2}((t-x)^2;x,y)}}{\delta_1}\right)\omega^1(f;\delta_1)+\left(1+\frac{\sqrt{K_{n,m,\alpha_{1},\alpha_2}^{\beta_1,\beta_2}((s-y)^2;x,y)}}{\delta_2}\right)\omega^2(f;\delta_2).    
\end{align*}
On choosing $\delta_1:=\delta^2_{n,\alpha_1,\beta_1}(x)$ and $\delta_2:=\delta^2_{m,\alpha_2,\beta_2}(y),$ we reach at the desired outcome.
		 \end{proof}
		 \begin{thm}
		 	For any $f\in C^2(I^2)$ and $(x,y)\in I^2$ we obtain:
		 	\begin{eqnarray*}
		 	\lim_{n\rightarrow \infty}n (K_{n,n,\alpha_{1},\alpha_2}^{\beta_1,\beta_2}(f;x,y)-f(x,y))&=&\left(-x+\dfrac{1}{\beta_1+1}\right)f_x(x,y)+\left(-y+\dfrac{1}{\beta_2+1}\right)f_y(x,y)\\
		 	&& +\frac{1}{2}\left[x(1-x)f_{xx}(x,y)+y(1-y)f_{yy}(x,y)\right].
		 	\end{eqnarray*}
		 \end{thm}
		 \begin{proof}
		 	Using the Taylor's formula for a fixed point $(x,y)\in I^2$
		 	\begin{eqnarray*}
		 		f(t,s)&=&f(x,y)+f_x(x,y)(t-x)+f_y(x,y)(s-y)+\frac{1}{2}\left[f_{xx}(x,y)(t-x)^2\right.\\
		 		&&+\left.2f_{xy}(x,y)(t-x)(s-y)+f_{yy}(x,y)(s-y)^2\right]+\Theta(t,s)\{(t-x)^2+(s-y)^2\},
		 	\end{eqnarray*}
		 	where $\displaystyle\lim_{(t,s)\rightarrow (x,y)} \Theta(t,s)=0$ and $(x,y)\in I^2.$\\
		 	Applying the operators and using its linearity property
		 	\begin{eqnarray}\label{eq9}
		 		n\left(K_{n,n,\alpha_{1},\alpha_2}^{\beta_1,\beta_2}(f;x,y)-f(x,y)\right)&=& n\bigg( f_x(x,y)K_{n,n,\alpha_{1},\alpha_2}^{\beta_1,\beta_2}(t-x;x,y)+f_y(x,y)J_{n,n,\rho_1,\rho_1}^{\alpha_1,\alpha_1}(s-y;x,y)\nonumber\\
		 		&&+\frac{1}{2}\left\{f_{xx}(x,y)K_{n,n,\alpha_{1},\alpha_2}^{\beta_1,\beta_2}((t-x)^2;x,y)\right.\nonumber\\
		 		&&+\left.2f_{xy}(x,y)K_{n,n,\alpha_{1},\alpha_2}^{\beta_1,\beta_2}((t-x)(s-y);x,y)\right.\nonumber\\
		 		&&+\left.f_{yy}(x,y)K_{n,n,\alpha_{1},\alpha_2}^{\beta_1,\beta_2}((s-y)^2;x,y)\right\}\nonumber\\
		 		&&+K_{n,n,\alpha_{1},\alpha_2}^{\beta_1,\beta_2}(\Theta(t,s)\{(t-x)^2+(s-y)^2\};x,y)\bigg).
		 	\end{eqnarray}
		 	Apply Holder's inequality to the last term of right hand side of (\ref{eq9})\\
		 	$n K_{n,n,\alpha_{1},\alpha_2}^{\beta_1,\beta_2}(\Theta(t,s)\{(t-x)^2+(s-y)^2\};x,y)$
		 	\begin{eqnarray}\label{eq10}
		 		&\leq & n \left[K_{n,n,\alpha_{1},\alpha_2}^{\beta_1,\beta_2}({\Theta}^2(t,s);x,y)\right]^{\frac{1}{2}}.
		 		\left[K_{n,n,\alpha_{1},\alpha_2}^{\beta_1,\beta_2}(((t-x)^2+(s-y)^2)^2;x,y)\right]^{\frac{1}{2}}\nonumber\\
		 		&\leq &\left[K_{n,n,\alpha_{1},\alpha_2}^{\beta_1,\beta_2}({\Theta}^2(t,s);x,y)\right]^{\frac{1}{2}}\nonumber\\
		 		&&\times \sqrt{2}n \left[K_{n,n,\alpha_{1},\alpha_2}^{\beta_1,\beta_2}((t-x)^4;x,y)+J_{n,n,\rho_1,\rho_1}^{\alpha_1,\alpha_1}((s-y)^4;x,y)\right]^{\frac{1}{2}}
		 	\end{eqnarray}
		 	{\begin{flushright}{[$\because$ $(a+b)^2\leq 2(a^2+b^2)$].}\end{flushright}}
		 \noindent	By using Theorem \ref{thm2}, we get
		 	$$\lim_{n\rightarrow \infty} K_{n,n,\alpha_{1},\alpha_2}^{\beta_1,\beta_2}({\Theta}^2(t,s);x,y)=0.$$
		 	With the help of central moments of order 4 in Lemma \ref{lemma3}, we get
		 	$$\displaystyle\lim_{n\rightarrow \infty} n\, K_{n,n,\alpha_{1},\alpha_2}^{\beta_1,\beta_2}(\Theta(t,s)\{(t-x)^2+(s-y)^2\};x,y)=0.$$
		 	Using Lemma \ref{lemma3}, (\ref{eq9}) and (\ref{eq10}), we get our required result.
		 	\end{proof}
		 
		 \begin{definition}
		 	For $0< \gamma_1,\gamma_2 \leq 1,$ the Lipshitz class denoted by $Lip_M(\gamma_1,\gamma_2)$ for two variables is defined as
		 	$$Lip_M(\gamma_1,\gamma_2):=\left\{f:C(I^2):|f(t,s)-f(x,y)|\leq M |t-x|^{\gamma_1}|s-y|^{\gamma_2}\right\},$$
		 	where $M>0$ and $(t,s),(x,y)\in I^2.$
		 \end{definition}
		 \begin{thm}
		 	For $f\in Lip_M(\gamma_1,\gamma_2),$ we have the following inequality:
		 	$$\mid K_{n,m,\alpha_{1},\alpha_2}^{\beta_1,\beta_2}(f;x,y)-f(x,y)\mid \leq M[\delta^2_{n,\alpha_1,\beta_1}]^{\frac{\gamma_1}{2}}[\delta^2_{m,\alpha_2,\beta_2}]^{\frac{\gamma_2}{2}},$$
		 		where $\delta_{n,\alpha_1,\beta_1}^2(x)$ and $\delta_{m,\alpha_2,\beta_2}^2(y)$ are as defined in Theorem \ref{complete}.
		 \end{thm}
		 \begin{proof}
		 	Let $f\in Lip_M(\gamma_1,\gamma_2).$ Consider
\begin{align*}
\mid K_{n,m,\alpha_{1},\alpha_2}^{\beta_1,\beta_2}(f;x,y)-f(x,y)\mid 
&\leq K_{n,m,\alpha_{1},\alpha_2}^{\beta_1,\beta_2}(\mid f(t,s)-f(x,y)\mid;x,y)\\
& \leq MK_{n,m,\alpha_{1},\alpha_2}^{\beta_1,\beta_2}(\mid t-x\mid^{\gamma_1}\mid s-y \mid^{\gamma_2};x,y)\\
&\leq  M K_{n,m,\alpha_{1},\alpha_2}^{\beta_1,\beta_2}(\mid t-x\mid^{\gamma_1};x,y)\,K_{n,m,\alpha_{1},\alpha_2}^{\beta_1,\beta_2}(\mid s-y\mid^{\gamma_2};x,y).
	\end{align*}
By H\"older's inequality for $p=\frac{2}{\gamma_1}, q=\frac{2}{2-\gamma_1},$ we have
		 	$$K_{n,m,\alpha_{1},\alpha_2}^{\beta_1,\beta_2}(\mid t-x\mid^{\gamma_1};x,y)\leq \left(K_{n,m,\alpha_{1},\alpha_2}^{\beta_1,\beta_2}((t-x)^2;x,y)\right)^{\frac{\gamma_1}{2}}.$$
		 	Similarly for $p=\frac{2}{\gamma_2}, q=\frac{2}{2-\gamma_2},$ we have
		 	$$K_{n,m,\alpha_{1},\alpha_2}^{\beta_1,\beta_2}(\mid s-y\mid^{\gamma_1};x,y)\leq \left(K_{n,m,\alpha_{1},\alpha_2}^{\beta_1,\beta_2}((s-y)^2;x,y)\right)^{\frac{\gamma_2}{2}}.$$
		 	Hence, the result.
		 \end{proof}
		 \begin{definition}{\bf{Second order modulus of smoothness:}}
		 	For $g\in C(I^2)$
		 	$$\omega_2(f;\delta)=\sup\{|f(x+2u,y+2v)-2f(x+u,y+v)+f(x,y)|;(x,y),(x+2u,y+2v)\in I^2\}.$$
		 \end{definition}
		 Also, we have the relationship between K-functional and second order modulus of continuity \cite{butzer1}:
		 \begin{eqnarray}\label{eq8}
		 	K(f;\delta)\leq C\left[\omega_2(f;\sqrt{\delta})+min (1,\delta)\|f\|\right].
		 \end{eqnarray}
		 \begin{thm}\label{thm6}
		 	For $f\in C(I^2),$ we get the error estimation in terms of first and second order modulus of continuity
		 	\begin{eqnarray*}
		 		\mid K_{n,n,\alpha_{1},\alpha_2}^{\beta_1,\beta_2}(f;x,y)-f(x,y)\mid &\leq & 4 C\left[ \omega_{2}\left(f;\frac{1}{2}\sqrt{\Theta_{n,m,\alpha_1,\alpha_2}^{\beta_1,\beta_2}(x,y)}\right)+\min \left(1,\frac{1}{4}\Theta_{n,m,\alpha_1,\alpha_2}^{\beta_1,\beta_2}(x,y)\right)\right]\\
		 		&&+\omega(f;\mu_{n,m}(x,y)),\\
		 		\mbox{where}\quad \Theta_{n,m,\alpha_1,\alpha_2}^{\beta_1,\beta_2}(x,y)&=&
		 	\left[{K}_{n,m,\alpha_{1},\alpha_2}^{\beta_1,\beta_2}((t-x)^2;x,y)+\left(\dfrac{nx}{n+1}+\dfrac{1}{(\beta_1+1)(n+1)}\right)^{2}\right.\\
		 	&&+\left.{K}_{n,m,\alpha_{1},\alpha_2}^{\beta_1,\beta_2}((s-y)^2;x,y)+\left(\frac{my}{m+1}+\frac{1}{(\beta_2+1)(m+1)}\right)^{2}\right]\\
		 		\mbox{and}\quad \mu_{n,m}(x,y)&=& \left(\left(\dfrac{nx}{n+1}+\dfrac{1}{(\beta_1+1)(n+1)}-x\right)^{2}+\left(\frac{my}{m+1}+\frac{1}{(\beta_2+1)(m+1)}-y\right)^{2}\right)^{\frac{1}{2}},
		 	\end{eqnarray*}
		 	where $C>0$ is a constant.
		 \end{thm}
		 \begin{proof}
		 	Firstly, we define an auxiliary operators for $(x,y)\in I^2$ in the following way:
		 	\begin{eqnarray*}
		 		\hat{K}_{n,m,\alpha_{1},\alpha_2}^{\beta_1,\beta_2}(f;x,y)&=&{K}_{n,m,\alpha_{1},\alpha_2}^{\beta_1,\beta_2}(f;x,y)+f(x,y)\\
		 		&&-f\left(\dfrac{nx}{n+1}+\dfrac{1}{(\beta_1+1)(n+1)},\dfrac{my}{m+1}+\dfrac{1}{(\beta_2+1)(m+1)}\right).
		 	\end{eqnarray*}
		 	By using this definition, we can get
		 	\begin{eqnarray*}
		 	\hat{K}_{n,m,\alpha_{1},\alpha_2}^{\beta_1,\beta_2}(t-x;x,y)=0;\qquad
		 		\hat{K}_{n,m,\alpha_{1},\alpha_2}^{\beta_1,\beta_2}(s-y;x,y)=0.
		 	\end{eqnarray*}
		 	For $g\in C^2(I^2),$ we consider\\
		 	$\mid K_{n,m,\alpha_{1},\alpha_2}^{\beta_1,\beta_2}(f;x,y)-f(x,y)\mid $
		 	\begin{eqnarray}\label{eq3}
		 		&\leq & \left| K_{n,m,\alpha_{1},\alpha_2}^{\beta_1,\beta_2}(f;x,y)-\hat{K}_{n,m,\alpha_{1},\alpha_2}^{\beta_1,\beta_2}(f;x,y) \right|+\left|\hat{K}_{n,m,\alpha_{1},\alpha_2}^{\beta_1,\beta_2}(f;x,y)-\hat{K}_{n,m,\alpha_{1},\alpha_2}^{\beta_1,\beta_2}(g;x,y) \right|\nonumber\\
		 		&&+\left| \hat{K}_{n,m,\alpha_{1},\alpha_2}^{\beta_1,\beta_2}(g;x,y)-g(x,y)\right|+\left| g(x,y)-f(x,y)\right|\nonumber\\
		 		&=&\left| f\left(\dfrac{nx}{n+1}+\dfrac{1}{(\beta_1+1)(n+1)}, \dfrac{my}{m+1}+\dfrac{1}{(\beta_2+1)(m+1)}\right)-f(x,y)\right| \nonumber \\
		 		&& +\left| \hat{K}_{n,m,\alpha_{1},\alpha_2}^{\beta_1,\beta_2}(f-g;x,y)\right|+\left| \hat{K}_{n,m,\alpha_{1},\alpha_2}^{\beta_1,\beta_2}(g;x,y)-g(x,y)\right|+\left| g(x,y)-f(x,y)\right|.
		 	\end{eqnarray}
		 	Now, by Taylor's polynomial for $g(t,s)\in C^2(I^2)$, we have
		 	\begin{eqnarray*}
		 		g(t,s)&=&g(x,y)+(t-x)\frac{\partial g(x,y)}{\partial x}+\int_x^t(t-u)\frac{{\partial}^2 g(u,y)}{\partial u^2}\, du+(s-y)\frac{\partial g(x,y)}{\partial y}+
		 		\int_y^s (s-v)\frac{{\partial}^2g(x,v)}{\partial v^2}\, dv.
		 	\end{eqnarray*}
		 	Applying the operators $\hat{K}_{n,m,\alpha_{1},\alpha_2}^{\beta_1,\beta_2}(. ;x,y)$ \\
		 	$\left| \hat{K}_{n,m,\alpha_{1},\alpha_2}^{\beta_1,\beta_2}(g;x,y)-g(x,y)\right| $
		 	\begin{eqnarray*}
		 		&\leq &\left| \hat{K}_{n,m,\alpha_{1},\alpha_2}^{\beta_1,\beta_2}\left(\int_x^t(t-u)\frac{{\partial}^2g(u,y)}{\partial u^2}\,du;x,y\right)\right|+\left| \hat{K}_{n,m,\alpha_{1},\alpha_2}^{\beta_1,\beta_2}\left(\int_y^s(s-v)\frac{{\partial}^2g(x,v)}{\partial v^2}\,dv;x,y\right)\right|\\
		 		&\leq &  {K}_{n,m,\alpha_{1},\alpha_2}^{\beta_1,\beta_2}\left(\left|\int_x^t\left| t-u \right| \left| \frac{{\partial}^2g(u,y)}{\partial u^2}\right|\,du\right|;x,y\right)\\
		 		&&+\left|\displaystyle\int_x^{\frac{nx}{n+1}+\frac{1}{(\beta_1+1)(n+1)}}\left| \dfrac{nx}{n+1}+\dfrac{1}{(\beta_1+1)(n+1)}-u \right| \left| \frac{{\partial}^2g(u,y)}{\partial u^2}\right|\,du\right|\\
		 		&& + {K}_{n,m,\alpha_{1},\alpha_2}^{\beta_1,\beta_2}\left(\left|\int_y^s\left| s-v \right| \left| \frac{{\partial}^2g(x,v)}{\partial v^2}\right|\,dv\right|;x,y\right)\\
		 		&&+\left|\int_y^{\frac{my}{m+1}+\frac{1}{(\beta_2+1)(m+1)}}\left| \frac{my}{m+1}+\frac{1}{(\beta_2+1)(m+1)}-v \right| \left| \frac{{\partial}^2g(x,v)}{\partial v^2}\right|\,dv\right|\\
		 		&\leq &\left[{K}_{n,m,\alpha_{1},\alpha_2}^{\beta_1,\beta_2}((t-x)^2;x,y)+\left(\dfrac{nx}{n+1}+\dfrac{1}{(\beta_1+1)(n+1)}\right)^{2}\right.\\
		 		&&+\left.{K}_{n,m,\alpha_{1},\alpha_2}^{\beta_1,\beta_2}((s-y)^2;x,y)+\left(\frac{my}{m+1}+\frac{1}{(\beta_2+1)(m+1)}\right)^{2}\right]\|g\|_{C^2(I^2)}\\&:=&\Theta_{n,m,\alpha_1,\alpha_2}^{\beta_1,\beta_2}(x,y)\|g\|_{C^2(I^2)}.
		 	\end{eqnarray*}
		 	By using the definition of operators $\hat{K}_{n,m,\alpha_{1},\alpha_2}^{\beta_1,\beta_2}(f;x,y)$, we obtain
		 	\begin{eqnarray*}
		 		\mid \hat{K}_{n,m,\alpha_{1},\alpha_2}^{\beta_1,\beta_2}(f;x,y)\mid \leq 3\|f\|_{C(I^2)}.
		 	\end{eqnarray*}
		 	Thus, the equation (\ref{eq3}) becomes:
		 	\begin{eqnarray*}
		 		\mid K_{n,m,\alpha_{1},\alpha_2}^{\beta_1,\beta_2}(f;x,y)-f(x,y)\mid &\leq & 4\|f-g\|_{C(I^2)}+\Theta_{n,m,\alpha_1,\alpha_2}^{\beta_1,\beta_2}(x,y)\|g\|_{C^2(I^2)}\\
		 		&&+\omega\left(f;\mu_{n,m}(x,y)\right)\\
		 		&=&4\left\{\|f-g\|_{C(I^2)}+\frac{1}{4}\Theta_{n,m,\alpha_1,\alpha_2}^{\beta_1,\beta_2}(x,y)\|g\|_{C^2(I^2)}\right\}+\omega(f;\mu_{n,m}(x,y)).
		 	\end{eqnarray*}
		 	Taking the infimum over $g\in C^2(I^2)$ and using the relation (\ref{eq8}), we get our desired result.
		 \end{proof}
		 \begin{thm}
		 	Let $f,g\in C^2(I^2),$ then we have the following identity:
		 	\begin{align*}
		 		\lim_{n\rightarrow \infty}n \left[K_{n,n,\alpha_{1},\alpha_2}^{\beta_1,\beta_2}(fg;x,y)-K_{n,n,\alpha_{1},\alpha_2}^{\beta_1,\beta_2}(f;x,y)K_{n,n,\alpha_{1},\alpha_2}^{\beta_1,\beta_2}(g;x,y)\right]=&x(1-x)f_x(x,y)g_x(x,y)\\
		 		&+y(1-y)f_y(x,y)g_y(x,y).
		 	\end{align*}
		 \end{thm}	
		 \begin{proof}
		 	By Taylor's series, we obtain
		 	\\ $	K_{n,n,\alpha_{1},\alpha_2}^{\beta_1,\beta_2}(fg;x,y)-K_{n,n,\alpha_{1},\alpha_2}^{\beta_1,\beta_2}(f;x,y)\,K_{n,n,\alpha_{1},\alpha_2}^{\beta_1,\beta_2}(g;x,y)$	
		 	\begin{align*}
		 		=& K_{n,n,\alpha_{1},\alpha_2}^{\beta_1,\beta_2}(fg;x,y)-(fg)(x,y)-(f_x(x,y)g(x,y)+f(x,y)g_x(x,y))K_{n,n,\alpha_{1},\alpha_2}^{\beta_1,\beta_2}(t-x;x,y)\\
		 		&-(f_y(x,y)g(x,y)+f(x,y)g_y(x,y))K_{n,n,\alpha_{1},\alpha_2}^{\beta_1,\beta_2}(s-y;x,y)-\dfrac{1}{2}(f_{xx}(x,y)g(x,y)+2f_x(x,y)g_x(x,y)\\
		 		&+f(x,y)g_{xx}(x,y))K_{n,n,\alpha_{1},\alpha_2}^{\beta_1,\beta_2}((t-x)^2;x,y)-\dfrac{1}{2}(f_{yy}(x,y)g(x,y)+2f_y(x,y)g_y(x,y)\\
		 		&+f(x,y)g_{yy}(x,y))K_{n,n,\alpha_{1},\alpha_2}^{\beta_1,\beta_2}((s-y)^2;x,y)-(f(x,y)g_{xy}(x,y)+f_y(x,y)g_x(x,y)\\
		 		&+f_x(x,y)g_y(x,y)+f_{xy}(x,y)g(x,y))K_{n,n,\alpha_{1},\alpha_2}^{\beta_1,\beta_2}((t-x)(s-y);x,y)\\
		 		&-g(x,y)\bigg[K_{n,n,\alpha_{1},\alpha_2}^{\beta_1,\beta_2}(f;x,y)-f(x,y)-f_x(x,y)K_{n,n,\alpha_{1},\alpha_2}^{\beta_1,\beta_2}(t-x;x,y)-f_y(x,y)K_{n,n,\alpha_{1},\alpha_2}^{\beta_1,\beta_2}(s-y;x,y)\\
		 		&-\dfrac{1}{2}f_{xx}(x,y)K_{n,n,\alpha_{1},\alpha_2}^{\beta_1,\beta_2}((t-x)^2;x,y)-\dfrac{1}{2}f_{yy}(x,y)K_{n,n,\alpha_{1},\alpha_2}^{\beta_1,\beta_2}((s-y)^2;x,y)\\
		 		&-f_{xy}(x,y)K_{n,n,\alpha_{1},\alpha_2}^{\beta_1,\beta_2}((t-x)(s-y);x,y)\bigg]-K_{n,n,\alpha_{1},\alpha_2}^{\beta_1,\beta_2}(f;x,y)\left[K_{n,n,\alpha_{1},\alpha_2}^{\beta_1,\beta_2}(g;x,y)-g(x,y)\right.\\
		 		&-\left.g_x(x,y)K_{n,n,\alpha_{1},\alpha_2}^{\beta_1,\beta_2}(t-x;x,y)-g_y(x,y)K_{n,n,\alpha_{1},\alpha_2}^{\beta_1,\beta_2}(s-y;x,y)-\dfrac{1}{2}g_{xx}(x,y)K_{n,n,\alpha_{1},\alpha_2}^{\beta_1,\beta_2}((t-x)^2;x,y)\right.\\
		 		&\left. -\dfrac{1}{2}g_{yy}(x,y)K_{n,n,\alpha_{1},\alpha_2}^{\beta_1,\beta_2}((s-y)^2;x,y)-g_{xy}(x,y)K_{n,n,\alpha_{1},\alpha_2}^{\beta_1,\beta_2}((t-x)(s-y);x,y)\right]\\
		 		& +g_x(x,y)K_{n,n,\alpha_{1},\alpha_2}^{\beta_1,\beta_2}(t-x;x,y)[f(x,y)-K_{n,n,\alpha_{1},\alpha_2}^{\beta_1,\beta_2}(f;x,y)]\\
		 		&+g_y(x,y)K_{n,n,\alpha_{1},\alpha_2}^{\beta_1,\beta_2}(s-y;x,y)[f(x,y)-K_{n,n,\alpha_{1},\alpha_2}^{\beta_1,\beta_2}(f;x,y)]\\
		 		&+\dfrac{1}{2}g_{xx}(x,y)K_{n,n,\alpha_{1},\alpha_2}^{\beta_1,\beta_2}((t-x)^2;x,y)[f(x,y)-K_{n,n,\alpha_{1},\alpha_2}^{\beta_1,\beta_2}(f;x,y)]\\
		 		&+\dfrac{1}{2}g_{yy}(x,y)K_{n,n,\alpha_{1},\alpha_2}^{\beta_1,\beta_2}((s-y)^2;x,y)[f(x,y)-K_{n,n,\alpha_{1},\alpha_2}^{\beta_1,\beta_2}(f;x,y)]\\
		 		& +g_{xy}(x,y)K_{n,n,\alpha_{1},\alpha_2}^{\beta_1,\beta_2}((t-x)(s-y);x,y)[f(x,y)-K_{n,n,\alpha_{1},\alpha_2}^{\beta_1,\beta_2}(f;x,y)]\\
		 		&+f_x(x,y)g_x(x,y)K_{n,n,\alpha_{1},\alpha_2}^{\beta_1,\beta_2}((t-x)^2;x,y)+f_x(x,y)g_y(x,y)K_{n,n,\alpha_{1},\alpha_2}^{\beta_1,\beta_2}((t-x)(s-y);x,y)\\
		 		&+f_y(x,y)g_x(x,y)K_{n,n,\alpha_{1},\alpha_2}^{\beta_1,\beta_2}((t-x)(s-y);x,y)+f_y(x,y)g_y(x,y)K_{n,n,\alpha_{1},\alpha_2}^{\beta_1,\beta_2}((s-y)^2;x,y)
		 	\end{align*}
		 	Now, applying the operators and taking limit as ${n\rightarrow \infty},$ we attain
		 	\begin{align*}
		 		\lim_{n\rightarrow\infty} n [K_{n,n,\alpha_{1},\alpha_2}^{\beta_1,\beta_2}(fg;x,y)-K_{n,n,\alpha_{1},\alpha_2}^{\beta_1,\beta_2}(f;x,y)K_{n,n,\alpha_{1},\alpha_2}^{\beta_1,\beta_2}(g;x,y)]=&x(1-x)f_x(x,y)g_x(x,y)\\
		 		&+y(1-y)f_y(x,y)g_y(x,y).
		 	\end{align*}
		 \end{proof}
	\section{Numerical Examples}\label{numerical}
	In this section, we will prove the convergence of the operators by choosing different values of the introduced shape parameters $\alpha_1, \alpha_2$ and $\beta_1,\beta_2$ and $n, m$ taking into account different examples. 
	\begin{example}
Let $f(x)=e^{-x} \sin^2(xy)$(red). In the Figure \ref{f1}, we show the convergence of the operators \eqref{op2} by choosing various values of $\alpha_1, \alpha_2.$ It represents the effect of the introduced shape parameters. The values of the parameters are as $n=m=25,$ $\beta_1=\beta_2=0.5,$ and $\alpha_1=0.8,\, \alpha_2=0.8$(blue), $\alpha_1=0.3,\, \alpha_2=0.4$ (green). Figure \ref{f2} represents the error of approximation of the operators \eqref{op2} from the given function $f(x,y)$ defined as $E_{n,m,\alpha_{1},\alpha_2}^{\beta_1,\beta_2}(f;x,y)=\mid K_{n,m,\alpha_{1},\alpha_2}^{\beta_1,\beta_2}(f;x,y)-f(x,y)\mid$ for the specific values of the parameter.
		\begin{figure}[h]
			\includegraphics[width=0.75\columnwidth]{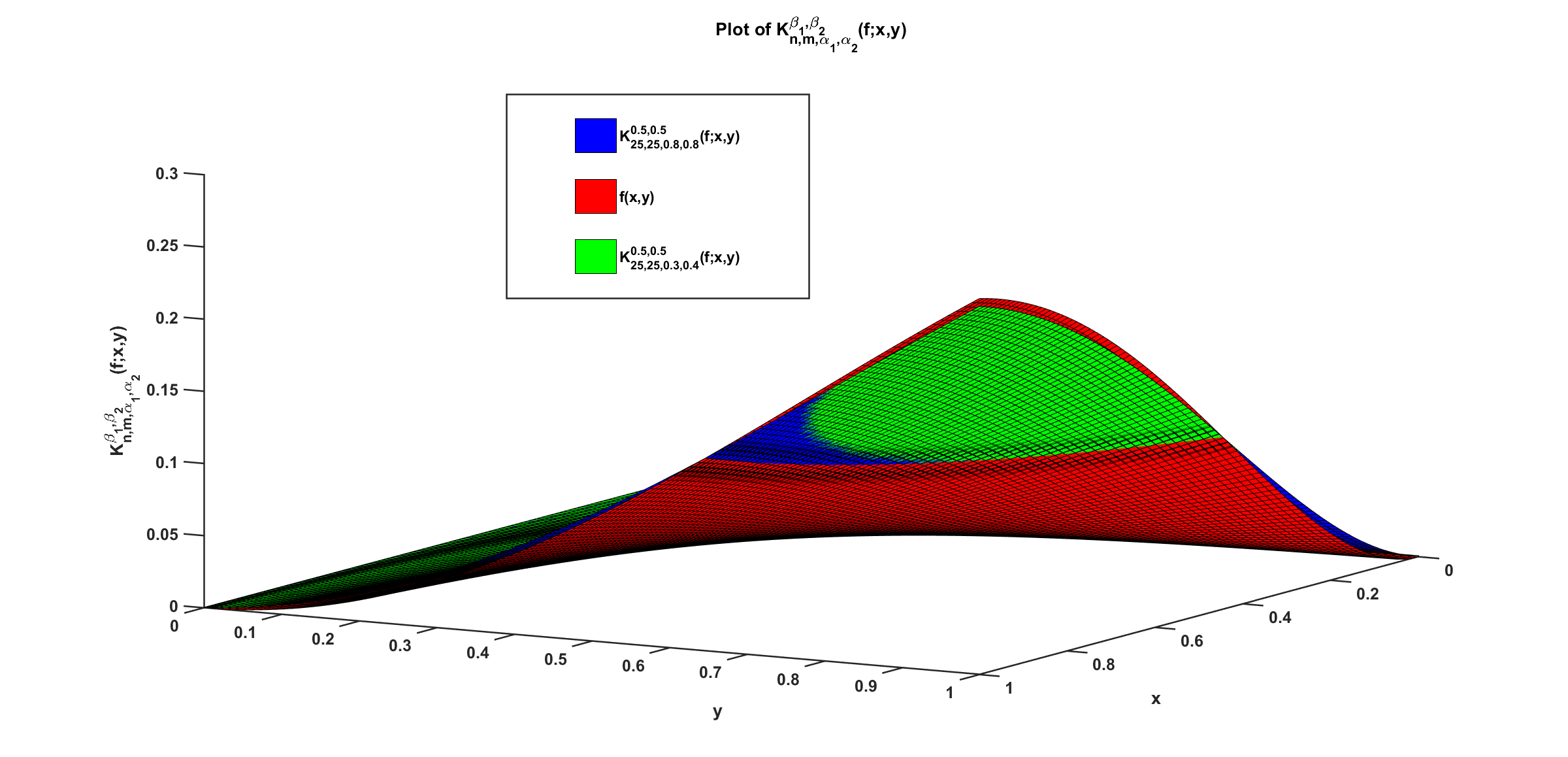}\\
			\caption{Convergence of the operators}
			\label{f1}
		\end{figure}
	\begin{figure}[h]
		\includegraphics[width=0.75\columnwidth]{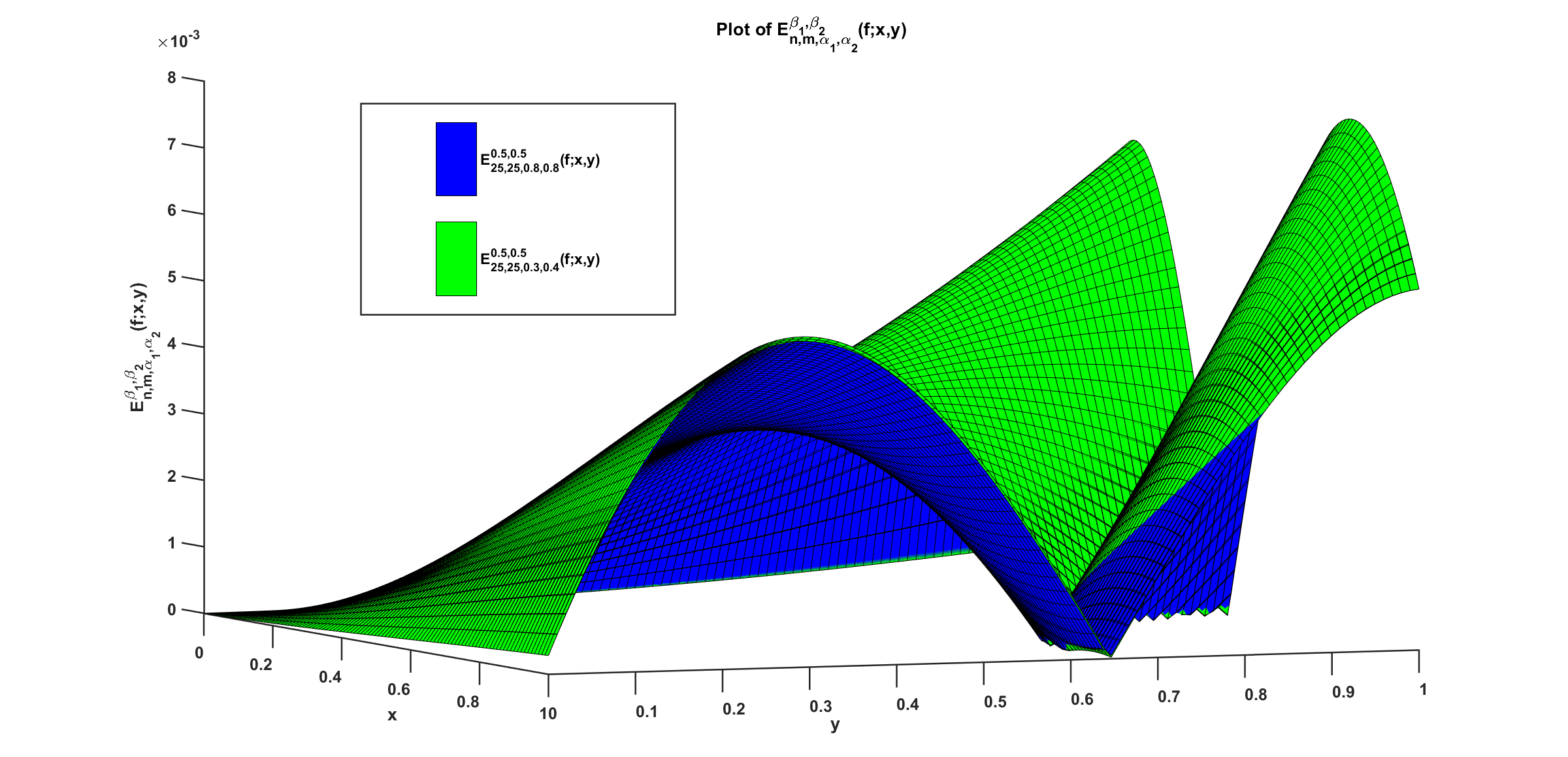}\\
		\caption{Error analysis of the operators}
		\label{f2}
	\end{figure}	
	\end{example}
\begin{example}
Consider $f(x,y)=x^4y^2 -2x^2y^3+xy^2$(red).In the Figure \ref{f3}, we show the convergence of the operators \eqref{op2} for the different values of $n=m=20$ (blue) and $n=m=30$ (green), where other parameters are fixed as $\alpha_1=0.5,\, \alpha_2=0.4,\, \beta_1=0.8, \,\beta_2=0.8.$ From the figure \ref{f3}, it is clear that as we increase the values of $n$ and $m$, the operators converge faster to the given function $f(x,y).$\\
	\begin{figure}[h]
		\includegraphics[width=0.75\columnwidth]{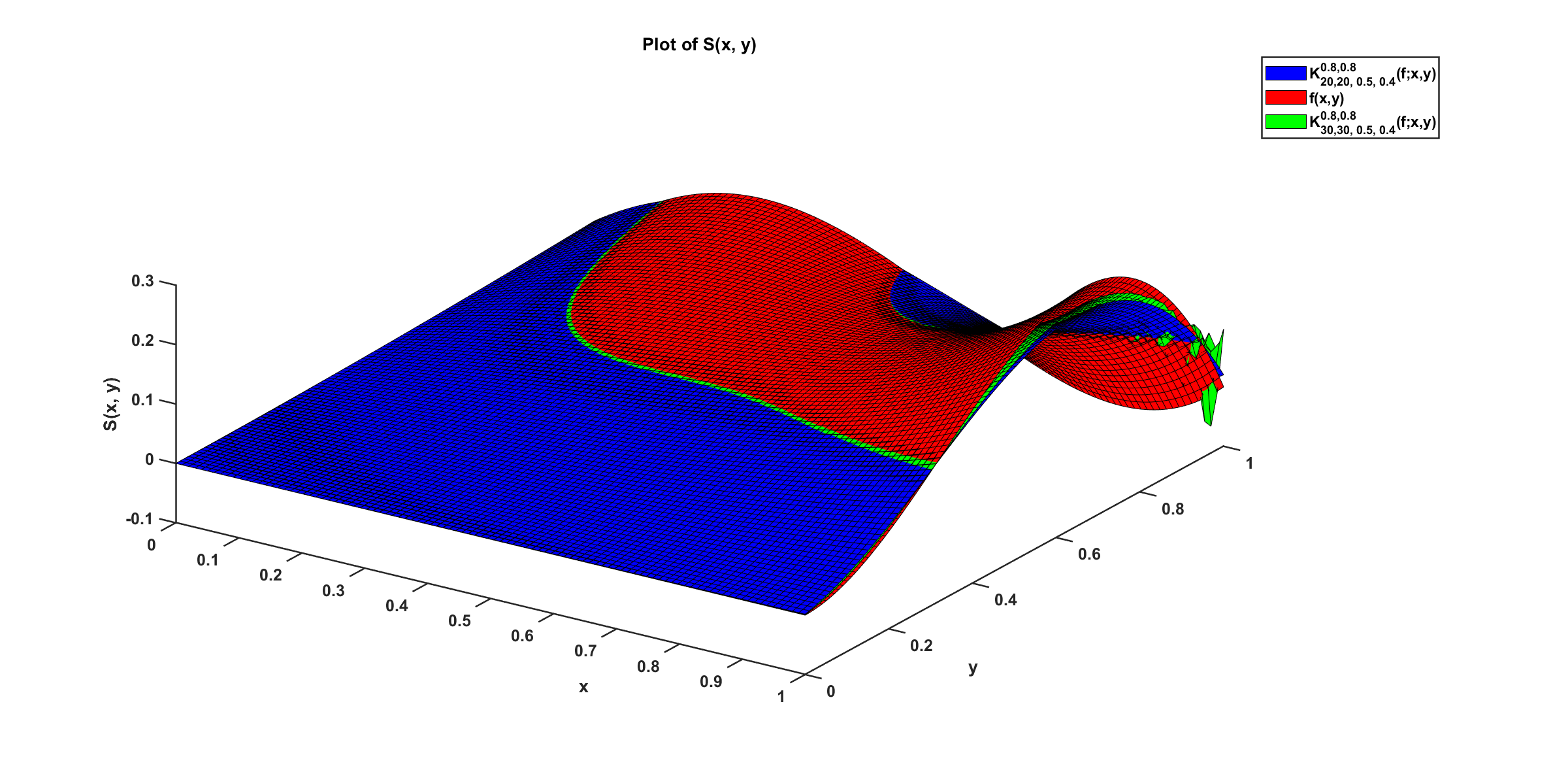}\\
		\caption{Convergence of the operators}
		\label{f3}
	\end{figure}
	Similarly, Figure \ref{f4} represents the error of approximation of the operators for the certain values of the parameters. The results from this figure are related to Figure \ref{f3} that is by increasing the values of $n$ and $m$, the induced error is converging to zero fastly.
	\begin{figure}[h]
		\includegraphics[width=0.75\columnwidth]{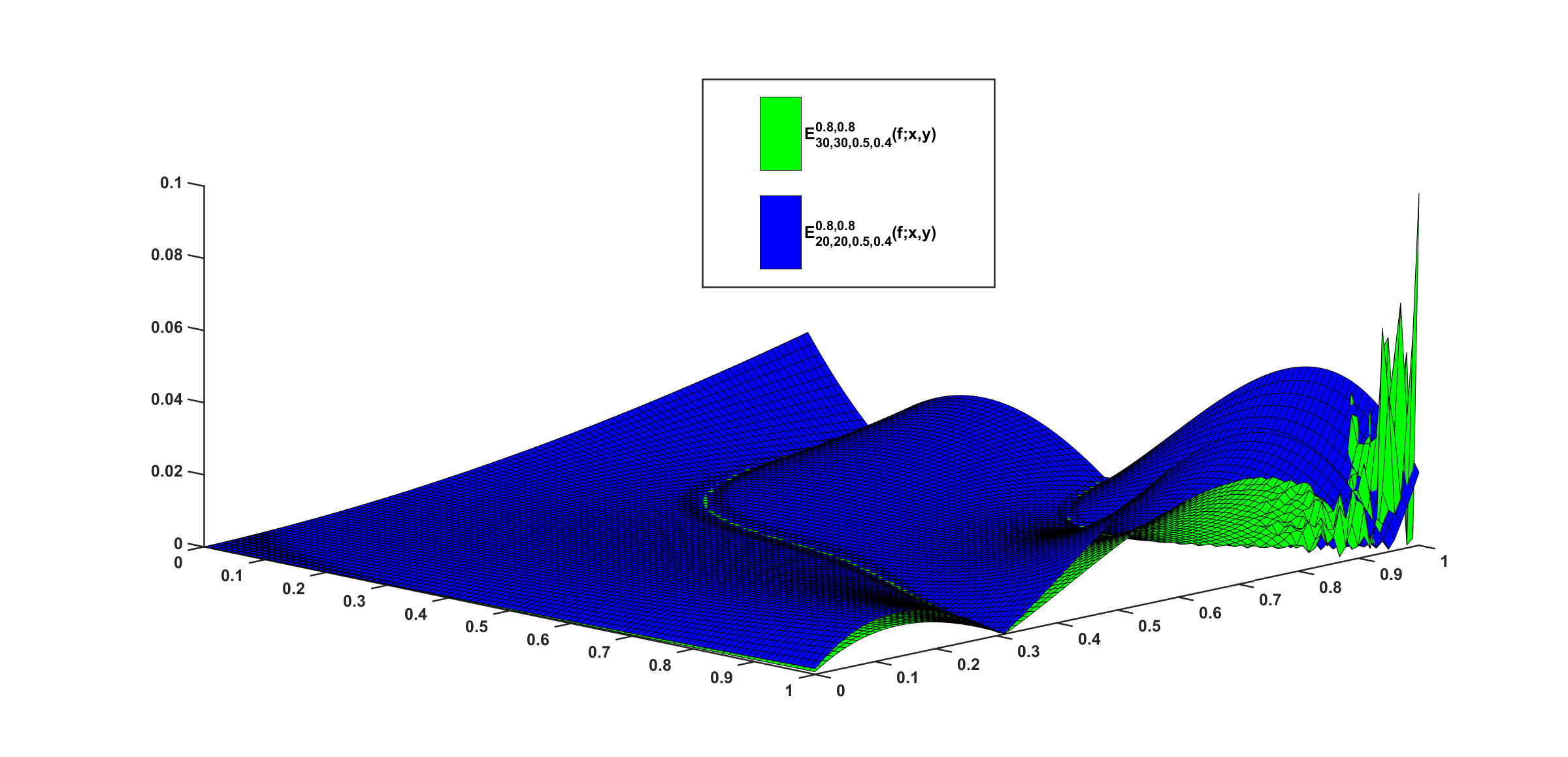}\\
		\caption{Error analysis of the operators}
		\label{f4}
	\end{figure}

	\end{example}
	\begin{example}
Let	$f(x,y)=y^2\cos(x)$(red). In the present example, we have studied the convergence of the operators \eqref{op2} by varying the introduced parameters $\beta_1$ and $\beta_2$. We consider $n=m=10,$ $\alpha_1=0.9, \,\alpha_2=0.8,$ and $\beta_1=0.5, \,\beta_2=0.4$(blue), $\beta_1=0.3, \beta_2=0.3$(green) in Figure \ref{f7}. Also for these certain parameters, we show the error of approximation in Figure \ref{f8}.\\
		\begin{figure}[h]
			\includegraphics[width=0.75\columnwidth]{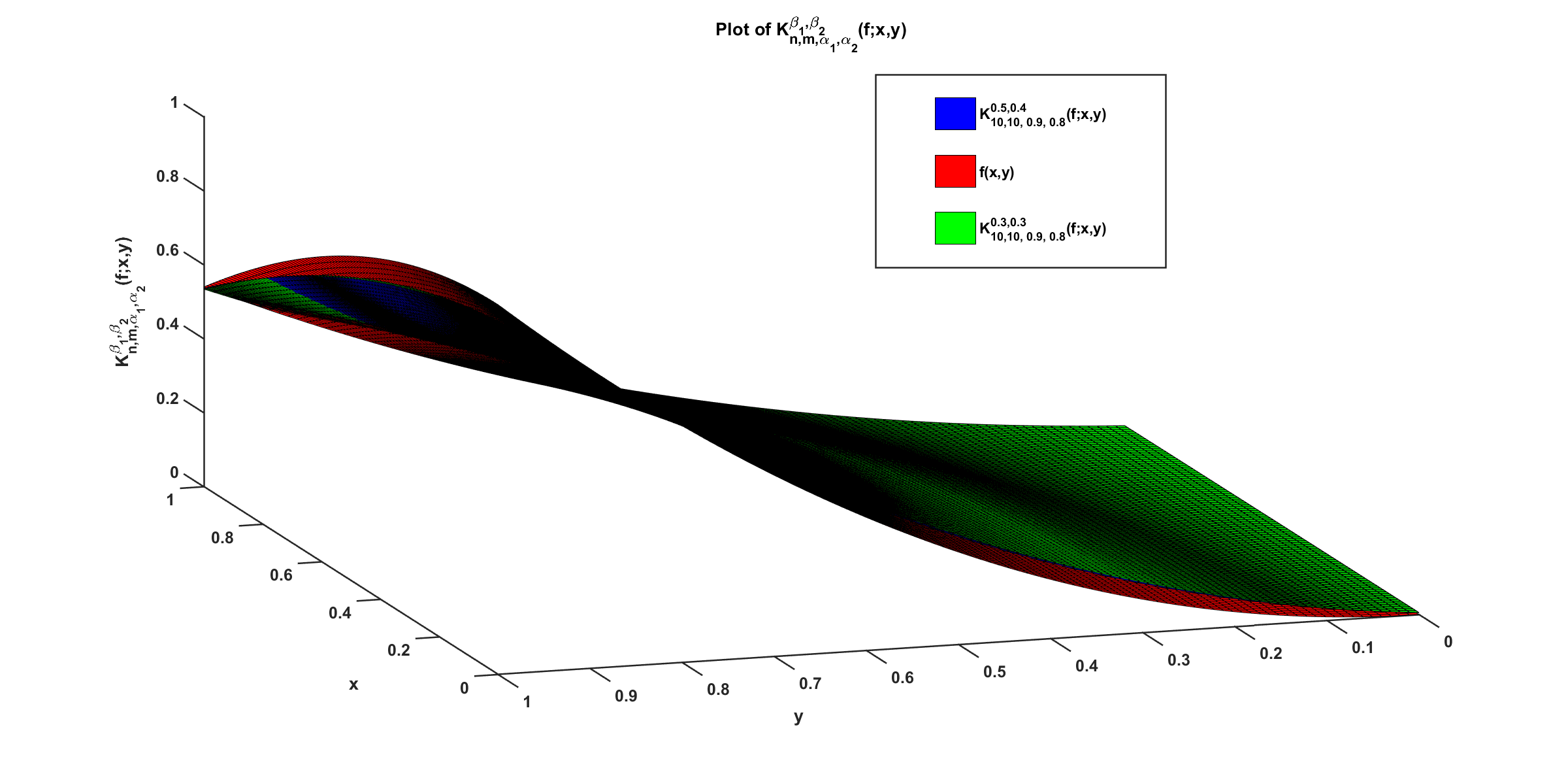}\\
			\caption{Convergence of the operators}
			\label{f7}
		\end{figure}
		\begin{figure}[h]
			\includegraphics[width=0.75\columnwidth]{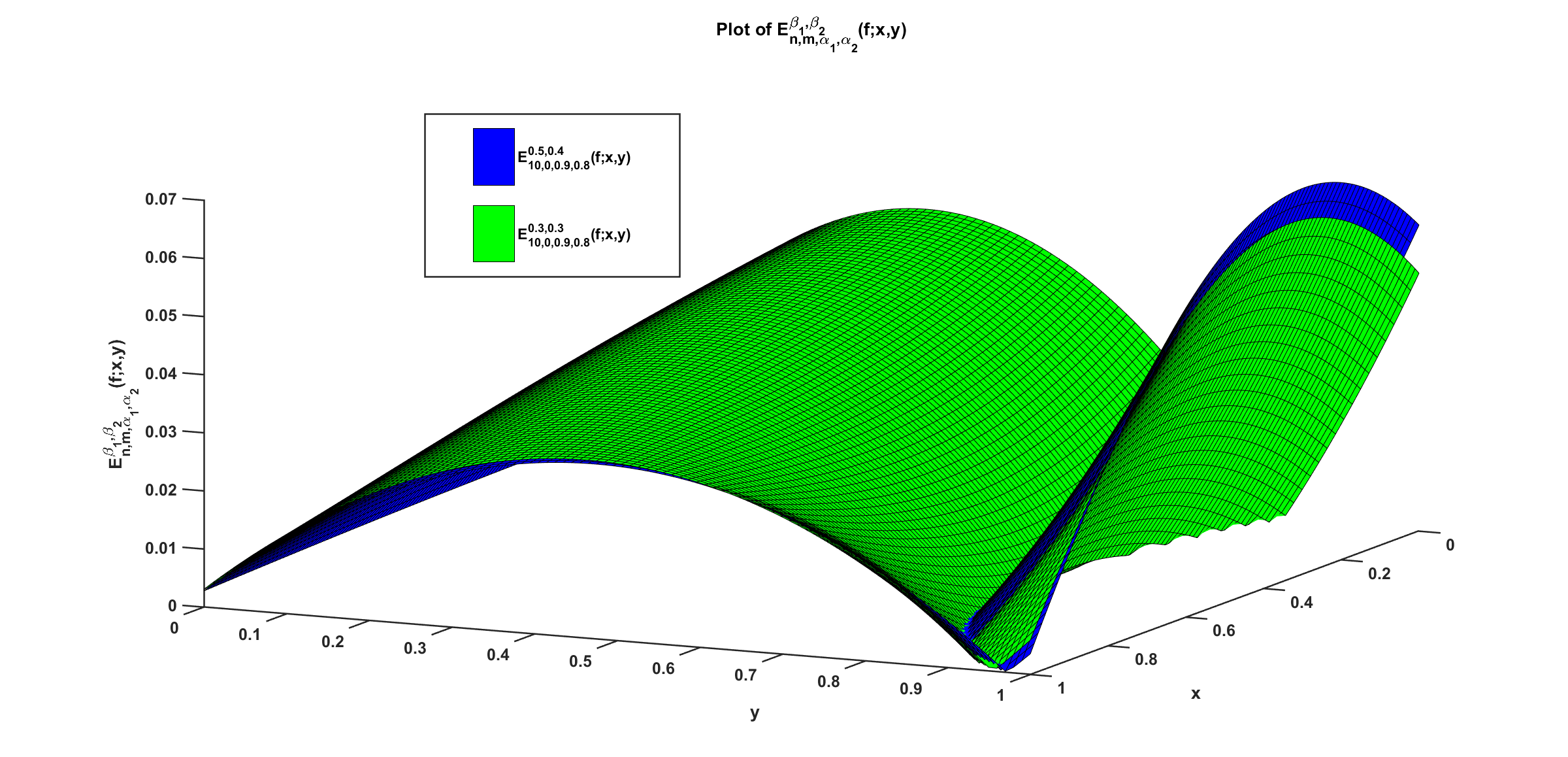}\\
			\caption{Error analysis of the operators}
			\label{f8}
		\end{figure}
	\end{example}
	{\textbf{Conclusion:}} The present article provides a new generalization of $\alpha-$Bernstein operators of Kantorovich form to preserve the linear functions. In order to approximate the functions of two variables, we also define the bivariate of these functions. Also, we present some numerical illustrations to justify our theoretical results. \\
	\textbf{Acknowledgment}
	The first author is thankful to the ``University Grants Commission (UGC) (1001/(CSIR-UGC NET DEC. 2017))" India for financial support to carry out the above research work.\\
	\textbf{Conflicts of interest Statement} The corresponding author, on behalf of all the authors, declares the absence of any conflicts of interest.\\
	\textbf{Availability of Data}
	This article is not subject to data sharing requirements.
\bibliographystyle{abbrv}
\bibliography{bib}

\begin{thebibliography}{10}

\bibitem{anastassiou2012approx}
G.~A. Anastassiou and S.~G. Gal.
\newblock {\em Approximation {T}heory: {M}oduli of {C}ontinuity and {G}lobal
  {S}smoothness {P}reservation}.
\newblock Springer Science \& Business Media, 2012.

\bibitem{ansari2022kant}
K.~J. Ansari.
\newblock On {K}antorovich variant of {B}askakov type operators preserving some
  functions.
\newblock {\em Filomat}, 36(3):1049--1060, 2022.

\bibitem{ansari2022num}
K.~J. Ansari, F.~{\"O}zger, and Z.~{\"O}demi{\c{s}}~{\"O}zger.
\newblock Numerical and theoretical approximation results for
  {S}churer--{S}tancu operators with shape parameter $\lambda$.
\newblock {\em Comput. Appl. Math.}, 41(4):181, 2022.

\bibitem{Baskakov1957}
V.~A. Baskakov.
\newblock An instance of a sequence of linear positive operators in the space
  of continuous functions.
\newblock 113(2):249--251, 1957.

\bibitem{baytuncc2023app}
E.~Baytun{\c{c}}, H.~Aktu{\u{g}}lu, and N.~I. Mahmudov.
\newblock Approximation properties of {R}iemann-{L}iouville type fractional
  {B}ernstein-{K}antorovich operators of order $\alpha$.
\newblock {\em Math. Found. Comput.}, 2023.

\bibitem{Berwalapp}
S.~Berwal, S.~A. Mohiuddine, A.~Kajla, and A.~Alotaibi.
\newblock Approximation by {R}iemann-{L}iouville type fractional
  $\alpha$-{B}ernstein-{K}antorovich operators.
\newblock {\em Math. Methods Appl. Sci.}, 47:8275--8288, 2024.

\bibitem{bushnaq2022comp}
S.~Bushnaq, K.~Shah, S.~Tahir, K.~J. Ansari, M.~Sarwar, and T.~Abdeljawad.
\newblock Computation of numerical solutions to variable order fractional
  differential equations by using non-orthogonal basis.
\newblock {\em AIMS Math}, 7(6):10917--10938, 2022.

\bibitem{butzer1}
P.~L. Butzer and H.~Berens.
\newblock {\em Semi-groups of {O}perators and {A}pproximation}, volume 145.
\newblock Springer Science \& Business Media, 2013.

\bibitem{cai2022stat}
Q.~B. Cai, K.~J. Ansari, M.~T. Ersoy, and F.~{\"O}zger.
\newblock Statistical blending-type approximation by a class of operators that
  includes shape parameters $\lambda$ and $\alpha$.
\newblock {\em Mathematics}, 10(7):1149, 2022.

\bibitem{lambe}
Q.~B. Cai, B.~Y. Lian, and G.~Zhou.
\newblock Approximation properties of $\lambda$-{B}ernstein operators.
\newblock {\em J. Inequal. Appl.}, 2018:1--11, 2018.

\bibitem{CHE}
X.~Chen, J.~Tan, Z.~Liu, and J.~Xie.
\newblock Approximation of functions by a new family of generalized {B}ernstein
  operators.
\newblock {\em J. Math. Anal. Appl.}, 450(1):244--261, 2017.

\bibitem{deo2020alpha}
N.~Deo and R.~Pratap.
\newblock $\alpha$-{B}ernstein--{K}antorovich operators.
\newblock {\em Afrika Mat. (3)}, 31(3-4):609--618, 2020.

\bibitem{DT}
Z.~Ditzian and V.~Totik.
\newblock Moduli of {S}moothness, 1987.

\bibitem{two}
A.~Gadjiev and A.~Ghorbanalizadeh.
\newblock Approximation properties of a new type {B}ernstein--{S}tancu
  polynomials of one and two variables.
\newblock {\em Appl. Math. Comput.}, 216(3):890--901, 2010.

\bibitem{GA}
V.~Gupta and R.~P. Agarwal.
\newblock {\em Convergence {E}estimates in {A}pproximation {T}heory},
  volume~13.
\newblock Springer, 2014.

\bibitem{kajla2018blend}
A.~Kajla and T.~Acar.
\newblock Blending type approximation by generalized {B}ernstein-{D}urrmeyer
  type operators.
\newblock {\em Miskolc Math Notes}, 19(1):319--336, 2018.

\bibitem{kajla2019general}
A.~Kajla and M.~Goyal.
\newblock Generalized {B}ernstein--{D}urrmeyer operators of blending type.
\newblock {\em Afrika Mat. (3)}, 30(7-8):1103--1118, 2019.

\bibitem{kan}
L.~V. Kantorovich.
\newblock Sur certains d{\'e}veloppements suivant les polyn{\^o}mes de la forme
  de {S}.
\newblock {\em Bernstein, I, II, CR Acad. URSS}, 563(568):595--600, 1930.

\bibitem{kilbas1993}
A.~A. Kilbas, O.~I. Marichev, and S.~G. Samko.
\newblock Fractional {I}ntegrals and {D}erivatives ({T}heory and
  {A}pplications), 1993.

\bibitem{QBER}
N.~I. Mahmudov and P.~Sabanc{\i}gil.
\newblock Some approximation properties of {L}upa\c{s} \,$q$-analogue of
  {B}ernstein operators.
\newblock {\em arXiv preprint arXiv:1012.4245}, 2010.

\bibitem{rahman2024esti}
S.~Rahman and K.~J. Ansari.
\newblock Estimation using a summation integral operator of exponential type
  with a weight derived from the $\alpha$-{B}askakov basis function.
\newblock {\em Math. Methods Appl. Sci.}, 47(4):2535--2547, 2024.

\bibitem{stan}
D.~D. Stancu.
\newblock Approximation of functions by a new class of linear polynomial
  operators.
\newblock {\em Rev. Roum. Math. Pores et Appl.}, 13(8):1173--1194, 1968.

\bibitem{szasz1950general}
O.~Sz$\acute{a}$sz.
\newblock Generalization of {S}. {B}ernstein’s polynomials to the infinite
  interval.
\newblock {\em J. Res. Natl. Bur. Stand.}, 45(3):239--245, 1950.

\end{thebibliography}
\end{document}